\documentclass[12pt,reqno]{amsart}
\usepackage{amssymb}
\usepackage{amsmath}
\usepackage{amsthm}
\usepackage{eucal}
\usepackage{color}
\usepackage{amsfonts}
\usepackage[pdftex]{graphicx}
\usepackage[T1]{fontenc}
\usepackage[notcite,notref]{showkeys}

\newcommand{\R}{\mathbb R}

\newcommand{\Z}{\mathbb Z}

\newcommand{\C}{\mathbb C}

\newcommand{\sgn}{\text{sgn}}

\newcommand{\be}{\begin{equation}}
\newcommand{\ee}{\end{equation}}
\newcommand{\bp}{\begin{proof}}
\newcommand{\ep}{\end{proof}}
\newcommand{\bel}{\begin{equation}\label}
\newcommand{\eeq}{\end{equation}}
\newtheorem{theorem}{Theorem}[section]
\newtheorem{proposition}[theorem]{Proposition}
\newtheorem{remark}[theorem]{Remark}

\newtheorem{lemma}[theorem]{Lemma}
\newtheorem{corollary}[theorem]{Corollary}

\numberwithin{equation}{section}
\begin{document}
\title[Unique continuation for the BO eq.]{Uniqueness Properties of Solutions to the Benjamin-Ono equation and related models.}
\author{C. E. Kenig}
\address[C. E. Kenig]{Department of Mathematics\\University of Chicago\\Chicago, Il. 60637 \\USA.}
\email{cek@math.uchicago.edu}

\author{G. Ponce}
\address[G. Ponce]{Department  of Mathematics\\
University of California\\
Santa Barbara, CA 93106\\
USA.}
\email{ponce@math.ucsb.edu}

\author{L. Vega}
\address[L. Vega]{UPV/EHU\\Dpto. de Matem\'aticas\\Apto. 644, 48080 Bilbao, Spain, and Basque Center for Applied Mathematics,
E-48009 Bilbao, Spain.}
\email{luis.vega@ehu.es}


\begin{abstract} We prove that if $u_1,\,u_2$ are solutions of the Benjamin-Ono equation  defined in 
$ (x,t)\in\R \times [0,T]$ which agree in an open set $\Omega\subset \R \times [0,T]$, then $u_1\equiv u_2$. We extend this uniqueness result to a general class of equations of Benjamin-Ono type in both  the initial value problem and the initial periodic boundary value problem. This class of 1-dimensional non-local models includes the intermediate long wave equation. Finally, we present a slightly stronger version of our uniqueness results for the Benjamin-Ono equation. 
\end{abstract}

\maketitle

\section{Introduction}

We consider the initial value problem (IVP) for the Benjamin-Ono (BO) equation
\be\label{BO} 
\begin{aligned}
\begin{cases}
& \partial_t u - \mathcal H \partial_x^2  u + u\partial_x u = 0,\qquad (x,t) \in  \, \R\times \R,\\
&u(x,0)=u_0(x), 
 \end{cases}
\end{aligned}
\ee
where $u=u(x,t) $ is a real-valued function, and $\mathcal H$ denotes the Hilbert transform
\be\label{H}
\begin{aligned}
\mathcal Hf(x)& :=\frac{1}{\pi}\,\rm{p.v.}\Big(\frac{1}{x}\ast f\Big)(x) \\
\\
&:=\frac{1}{\pi}\lim_{\epsilon\downarrow 0}\int_{|y|>\epsilon} \frac{f(x-y)}{y}dy=(-i\,\sgn(\xi)\widehat{f}(\xi))^{\vee}(x)
\end{aligned}
\ee
\medskip

The BO equation was first deduced by Benjamin \cite{Be} and Ono \cite{On} as a model for long internal gravity waves in deep stratified fluids. Later, it was shown to be a completely integrable system (see \cite {AbFo}, \cite {CoWi} and references therein). In particular, real solutions of the IVP \eqref{BO} satisfy infinitely many conservation laws, which provide an a priori estimate for the $H^{n/2}$-norm, $n\in \Z^+$.
\vskip.1in

The problem of finding  the minimal regularity measured in the Sobolev scale $\,H^s(\R)$, $\,s\in\R,$ 
required to guarantee that the IVP \eqref{BO} is locally or globally well-posed (WP) in $\,H^s(\R)$ has been extensively studied, see \cite{ABFS}, \cite{Io1}, \cite{Po}, \cite{KoTz1}, \cite{KeKo}, \cite{Ta}, \cite{BuPl} and \cite{IoKe} where global WP was established in $H^0(\R) = L^2(\R)$, (for further details and results regarding the well-posedness of the IVP \eqref{BO}  we refer to \cite{MoPi} and to \cite{IT} for a different proof of the result in \cite{IoKe}). 

We remark that a result  established in \cite{MoSaTz} (see also \cite{KoTz2}) implies that no well-posedness result in $H^s(\R), s\in\R$, for the IVP \eqref{BO} can be established by using  solely a contraction principle argument.

\medskip
It was first shown in  \cite{Io1} and \cite{Io2} that polynomial decay of the data may not be preserved by the solution flow of the BO equation. The results  in \cite{Io1} and  \cite{Io2} which present some unique continuation properties of the BO equation have been extended to fractional order weighted Sobolev spaces and have shown to be optimal in \cite{FP}  and \cite{FLP}. More precisely, using the notation 
$$
Z_{s,r}:=H^s(\R)\cap L^2(|x|^{2r}dx),\;\dot{Z}_{s,r}=Z_{s,r}\cap \{f\in L^1(\R):\widehat {f}(0)=0\},
$$
with $\,s,r>0$ one has the  results :\newline

(i) \cite{FP} The IVP \eqref{BO} is locally WP in $Z_{s,r}$ for $s\geq r\in [1,5/2)$ and if $\,u\in C([0,T]: Z_{5/2,2})$ is a solution of \eqref{BO}      s.t.  $u(\cdot, t_j)\in Z_{5/2,5/2}$, $\,j=1,2$ with $t_1,\,t_2\in [0,T],\,t_1\neq t_2$, then
$\,u\in C([0,T]: \dot Z_{5/2,2})$.\newline

(ii) \cite{FP} The IVP \eqref{BO} is locally WP in $\dot Z_{s,r}$  $s\geq r\in [5/2,7/2)$.\newline

(iii) \cite{FP}  If $\,u\in C([0,T]:\dot Z_{7/2,3})$ is a solution of \eqref{BO} s.t.    $\,\exists\, t_1,\,t_2, \,t_3\in [0,T],  \,t_1<t_2<t_3$  with $u(\cdot, t_j)\in Z_{7/2,7/2},\,j=1,2,3 $, then $\,u\equiv 0$.\newline

(iv) \cite{FLP}  The IVP \eqref{BO} has solutions $\,u\in C([0,T]:\dot Z_{7/2,3}),\,u\not\equiv 0,$ for which   $\,\exists\, t_1,\,t_2,\in [0,T],  \,t_1<t_2,$ with $u(\cdot, t_j)\in Z_{7/2,7/2}, \,j=1,2$.

\vskip.1in 
Our first main  result in this work is the following theorem:

\medskip

\begin{theorem}\label{TH1}
Let $u_1,\;u_2$ be solutions to the IVP \eqref{BO} for $ (x,t)\in \R\times [0,T]$ such that
 \begin{equation}
 \label{m1}
 u_1,\;u_2\in C([0,T]:H^s(\R)) \cap C^1((0,T):H^{s-2}(\R)),\;\;s>5/2.
\end{equation}
If there exists an open set $\,\Omega \subset \R\times [0,T]$ such that
\be\label{H1a}
u_1(x,t)=u_2(x,t),\;\;\;\;(x,t)\in\Omega,
\ee
then,
\be\label{result1}
u_1(x,t)=u_2(x,t),\;\;\;\;(x,t)\in \R\times [0,T].
\ee

In particular, if $u_1$ vanishes in $\Omega$, then $u_1\equiv 0$.
\end{theorem}

\begin{remark}
\label{rr1}

(i) Under the same hypotheses, Theorem \ref{TH1} applies to solutions of the generalized BO equation
\be\label{gBO} 
\begin{aligned}
 \partial_t u - \mathcal H \partial_x^2  u + \partial_x f(u) = 0,\qquad (x,t) \in & ~ \R\times \R,
\end{aligned}
\ee
with $f:\R\to\R$  smooth enough and $f(0)=0$. In particular, it applies for $\,f(u)=u^k,\;k=2,3,4,...$ for which the well posedness of the associated IVP was considered in \cite{ABFS}, \cite{KeTa}, \cite{KeKo}, \cite{KPV1}, \cite{Vent1}, \cite{Vent2}, see also \cite{LiPo}. 
\vskip.1in

(ii) The hypothesis \eqref{m1} guarantees that the solutions satisfy the equation \eqref{BO} point-wise, which will be required in our proof.\vskip.1in

(iii) A similar result to that described in Theorem \ref{TH1}  for the IVP associated to the generalized Korteweg-de Vries equation
\be\label{gKdV} 
\begin{aligned}
 \partial_t u + \partial_x^3 u + \partial_x u^{k} = 0,\qquad (x,t) \in & ~ \R\times \R,\;k=2,3,....,
\end{aligned}
\ee
was established in \cite{SaSc}, and for some evolution equations of Schr\"odinger type in \cite{Iz}. In both cases, their proofs are based on appropriate forms of the so called Carleman estimates. Our proof of Theorem \ref{TH1} is elementary and relies on simple properties of the Hilbert transform as a boundary value of analytic functions.
\vskip.1in
(iv) We observe that the unique continuation in (iii) before the statement of Theorem \ref{TH1} applies to a single solution of the BO equation but not to any two solutions as in Theorem \ref{TH1}. This is due to the fact that the argument in the  proof there depends upon the whole symmetry structure of the BO equation. 

\vskip.1in
(v) Theorem \ref{TH1} can be seen as a corollary of the following linear result whose proof is exactly  the one given below for Theorem \ref{TH1} :

Assume that $\,k,\,j\in \Z^+\cup\{0\}$ and that 
$$
a_m :\R\times [0,T]\to\R,\;m=0,1,..., k,\;\;\,\text{and}\,\;\;\,b:\R\times [0,T]\to\R
$$
are  continuous functions with  $\,b(\cdot)$ never  vanishing on $(x,t)\in \R\times [0,T],$ and consider the IVP
\be\label{general} 
\begin{aligned}
\begin{cases}
&\displaystyle \partial_t w - b(x,t)\,\mathcal H\partial_x^j  w +\sum_{m=0}^ka_m(x,t)\partial_x^mw= 0,\\
\\
&w(x,0)=w_0(x).
\end{cases}
\end{aligned}
\ee
\end{remark}

\begin{theorem}\label{TH3}
 Let $$\,w\in C([0,T]:H^s(\R)) \cap C^1((0,T):H^{s-2}(\R)),\;\,\;\;s>\max\{k;j\}+1/2,$$ be a solution to the IVP  \eqref{general}. If there exists an open set $\,\Omega \subset \R\times [0,T]$ such that
\be\label{HA1}
w(x,t)=0,\;\;\;\;(x,t)\in\Omega,
\ee
then,
\be\label{result22}
w(x,t)=0\;\;\;\;(x,t)\in \R\times [0,T].
\ee
\end{theorem}

\vskip.1in
\begin{remark}\label{ole}
(i)  In particular, applying Theorem \ref{TH3} to the difference of two solutions $\,u_1,\,u_2\,$ of  the  Burgers-Hilbert (BH) equation (see \cite{BiHu})
  \be
  \label{BH}
   \partial_t u - \mathcal H   u + u\partial_x u = 0,\qquad (x,t) \in  \, \R\times \R,
  \ee
   one sees that the result in Theorem \ref{TH1}, with $\,s>3/2$, holds for the IVP associated to the BH equation \eqref{BH}.
\vskip.1in

(ii) The result of Theorem \ref{TH1} extends to solutions of the  initial periodic boundary value problem (IPBVP) associated to the generalized BO equation
\be\label{gBO-PBVP} 
\begin{aligned}
\begin{cases}
& \partial_t u - \mathcal H \partial_x^2  u + \partial_x f(u) = 0,\qquad (x,t) \in \mathbb S^1\times \R,\\
&u(x,0)=u_0(x),
\end{cases}
\end{aligned}
\ee
with $ \,f(\cdot)\,$ as in part (i) of this remark. More precisely :

\end{remark}
\medskip

\begin{theorem}\label{TH2}
Let $u_1,\;u_2$ be solutions of the IPBVP  \eqref{gBO-PBVP} in $ (x,t)\in \mathbb S^1\times [0,T]$ such that
 \begin{equation}
 \label{m2}
 u_1,\;u_2\in C([0,T]:H^s(\mathbb S^1)) \cap C^1((0,T):H^{s-2}(\mathbb S^1)),\;s>5/2.
\end{equation}
If there exists an open set $\,\Omega \subset \mathbb S^1\times [0,T]$ such that
\be\label{HB1}
u_1(x,t)=u_2(x,t),\;\;\;\;(x,t)\in\Omega,
\ee
then,
\be\label{result}
u_1(x,t)=u_2(x,t),\;\;\;\;(x,t)\in \mathbb S^1\times [0,T].
\ee

In particular, if $u_1$ vanishes in $\Omega$, then $u_1\equiv 0$.
\end{theorem}
\vskip.1in

\begin{remark}
The well-posedness of the initial IPBVP \eqref{gBO-PBVP} has been studied in \cite{Mo1}, \cite{Mo2} and  \cite{MoRi3}.
\end{remark}
\bigskip

Next, we consider the Intermediate Long Wave (ILW) equation
\be\label{ILW} 
\begin{aligned}
 \partial_t u - \mathcal L_\delta \partial_x^2  u + \frac1{\delta}\partial_x u  + u\partial_x u = 0,\qquad (x,t) \in & ~ \R\times \R,
\end{aligned}
\ee
where $u=u(x,t)$  is a real-valued function, $\,\delta>0\,$ and
\be\label{T}
\mathcal L_\delta f(x) :=-\frac1{2\delta}\,\rm{p.v.} \int \rm{coth}\it \left(\frac{\pi(x-y)}{2\delta}\right)f(y)dy.
\ee
Note that $\mathcal L_\delta$ is a multiplier operator with $\partial_x\mathcal L_\delta$ having symbol
\be
\label{symbol}
\sigma(\partial_x\mathcal L_\delta)=\widehat{\partial_x\mathcal L_\delta} =2\pi \xi \,\rm{coth}\,(2\pi \delta \xi).
\ee 
The ILW equation \eqref{ILW} describes  long internal gravity waves in a stratified fluid with finite depth represented by the parameter $\,\delta$, see \cite{KKD}, 
 \cite{Jo}, \cite{JE}.
 
 Also, the ILW equation has been proven to be complete integrable, see \cite{KSA} and \cite{KAS}.
 
In \cite{ABFS} it was proven that solutions of the ILW  as $\delta \to \infty$ (deep-water limit) converge to solutions of the BO equation with the same initial data.
 
Also, in \cite{ABFS} it was shown that  if $u_{\delta}(x,t)$ denotes the solution of the ILW equation \eqref{ILW}, then
\be
\label{scaleKdV}
v_{\delta}(x,t)=\,\frac{3}{\delta} \,u_{\delta}\big(x,\frac{3}{\delta} t\Big)
\ee
converges  as $\delta\to 0$ (shallow-water limit) to the solution  of the KdV equation, i.e.  \eqref{gKdV} with $k=2$, 
with the same initial data.

 For further comments on general properties of the ILW equation we refer to the recent survey \cite{Sa} and references therein.

The well-posedness of the IVP associated to the ILW equation \eqref{ILW} was studied in \cite{ABFS} and more recently in  \cite{MoVe}.

Our next theorem extends the result in Theorem \ref{TH1}  to solution of the IVP associated to the ILW\eqref{ILW}:

\begin{theorem}\label{TH5}
Let $u_1,\;u_2$ be solutions to \eqref{ILW} in $ (x,t)\in \R\times [0,T]$ such that
 \begin{equation}
 \label{m1a}
 u_1,\;u_2\in C([0,T]:H^s(\R)) \cap C^1((0,T):H^{s-2}(\R)),\;\;s>5/2.
\end{equation}
If there exists an open set $\,\Omega \subset \R\times [0,T]$ such that
\be\label{H1}
u_1(x,t)=u_2(x,t),\;\;\;\;(x,t)\in\Omega,
\ee
then,
\be\label{result2}
u_1(x,t)=u_2(x,t),\;\;\;\;(x,t)\in \R\times [0,T].
\ee

In particular, if $u_1$ vanishes in $\Omega$, then $u_1\equiv 0$.
\end{theorem}

\begin{remark}

The observations in (i) and (v) in Remark \ref{rr1} and (ii) in Remark \ref{ole} apply, after  some simple modifications, to the ILW equation \eqref{ILW}.
\end{remark}
\vskip.1in

Next, we present the following slight improvement of Theorem \ref{TH1} and Theorem \ref{TH2} :

\begin{theorem}\label{TH10}
Let $u_1,\;u_2$ be solutions to \eqref{BO} in $ (x,t)\in \R\times [0,T]$ such that
 \begin{equation}
 \label{m10}
 u_1,\;u_2\in C([0,T]:H^s(\R)) \cap C^1((0,T):H^{s-2}(\R)),\;\;s>5/2.
\end{equation}
If there exists an open set $\,I \subset \R,\;0\in I\,$ such that
\be\label{H10}
u_1(x,0)=u_2(x,0),\;\;\;\;\;\;\;\;x\in I,
\ee
and for each $\,N\in\Z^+$
\begin{equation}
\label{H15}
\int_{|x|\leq R}\;|\partial_tu_1(x,0)-\partial_tu_2(x,0)|^2dx\leq c_N\,R^N\;\;\;\;\;\;\text{as}\;\;\;\;\;R \downarrow \,0,
\end{equation}
then,
\be\label{result11}
u_1(x,t)=u_2(x,t),\;\;\;\;(x,t)\in \R\times [0,T].
\ee

\end{theorem}

\medskip

\begin{theorem}\label{TH11}
Let $u_1,\;u_2$ be solutions of the IPBVP  \eqref{gBO-PBVP} in $ (x,t)\in \mathbb S^1\times [0,T]\simeq \R/\Z\times[0,T]$ such that
 \begin{equation}
 \label{m22}
 u_1,\;u_2\in C([0,T]:H^s(\mathbb S^1)) \cap C^1((0,T):H^{s-2}(\mathbb S^1)),\;s>5/2.
\end{equation}
If there exists an open set $\,I \subset [-1/2,1/2]$ with $\,0\in I$ such that
\be\label{HB12}
u_1(x,0)=u_2(x,0),\;\;\;\;x\in I,
\ee
and for each $\,N\in\Z^+$
\begin{equation}
\label{H20}
\int_{|x|\leq R}\,\;|\partial_tu_1(x,0)-\partial_tu_2(x,0)|^2d\theta\leq c_N\,R^N\;\;\;\;\;\;\text{as}\;\;\;\;\;R \downarrow \,0,
\end{equation}

then,
\be\label{result23}
u_1(x,t)=u_2(x,t),\;\;\;\;(x,t)\in \mathbb S^1\times [0,T].
\ee

\end{theorem}
\vskip.1in

\begin{remark}
It will be clear form our proof of Theorem \ref{TH10} that a similar argument provides the proof of Theorem \ref{TH11} which will be omitted.
\end{remark}

The rest of this paper is organized as follows : section 2 contains some preliminary estimates required for Theorem \ref{TH1} as well as its proof. It also includes the modification needed to extend the argument in the proof of Theorem \ref{TH1} from the IVP to the IPBVP to prove Theorem \ref{TH2}. Section 3 contains the proof of Theorem \ref{TH5}, and section 4 consists of the proof of Theorem \ref{TH10}.

\section{Proof of Theorem \ref{TH1}}

To prove Theorem \ref{TH1} we need the following result from complex analysis whose proof follows directly from Schwarz reflection principle:

\begin{proposition}
\label{pro1}

Let $I\subseteq \R$ be an open interval, $\,b\in(0,\infty]$  and
\be\label{sets}
D_b=\{z=x+iy\in \C:0<y<b\},\;\;L=\{x+i0\in \C:x\in I\}.
\ee
Let $F:D_b\cup L\to\C$ be a continuous function such that $\,F \big|_{D_b}$ is analytic. If $\,F \big|_{L}\equiv 0$, then $\,F\equiv 0$.
\end{proposition}

As a consequence we have

\begin{corollary}
\label{col1}
Let $f\in H^s(\R),\,s>1/2$ be a real valued function. If there exists an open set $I\subset \R$ such that $$f(x)=\mathcal Hf(x)=0,\;\;\;\;\;\;\;\forall \,x\in I,
$$
then $f\equiv 0$.

\end{corollary}

\begin{proof}
Denoting $U=U(x,y)$ the harmonic extension  of $f$ to the upper half-plane $D$, one sees that its harmonic conjugate $V=V(x,y)$ has boundary value $V(x,0)=\mathcal Hf(x)$ with 
\be
\label{HT}
(\widehat{f+i\mathcal H f})(\xi)=2\,\chi_{[0,\infty)}(\xi)\,\widehat{f}(\xi),\;\;\;\;\;\;\;\;\widehat{f}\in L^1(\R).
\ee
Thus, $F:=U+iV$ is continuous on $\overline D_{\infty}$ and analytic on $D_{\infty}$ with $\,F \big|_{L}\equiv 0$. Hence, Proposition \ref{pro1} yields the desired result

\end{proof}

\begin{proof} [Proof of Theorem \ref{TH1} ] Defining $w(x,t)=(u_1-u_2)(x,t)$ one has  that
\be\label{eq1}
\partial_tw-\mathcal H \partial_x^2 w+\partial_xu_2\,w+u_1\,\partial_xw=0,\;\;\;(x,t)\in\R\times[0,T].
\ee

By hypotheses  \eqref{m1} and \eqref{H1} there exist open intervals $I,\,J\subset \R$ such that
\be\label{zeros1}
\begin{aligned}
w(x,t)&=\partial_xw(x,t)\\
&=\partial_tw(x,t)=\partial_x^2w(x,t)=0,\;\;\;\;\;\;\;\;\;\;(x,t)\in I\times J\subset \Omega.
\end{aligned}
\ee

Thus, the equation \eqref{eq1} tells us
\be\label{zeros2}
\mathcal H \partial^2_xw(x,t)=0,\;\;(x,t)\in I\times J\subset \Omega.
\ee

Combining \eqref{zeros1} and \eqref{zeros2} and fixing $t^*\in J$  it follows that 
\be\label{zeros3}
\partial_x^2w(x,t^*)=\mathcal H \partial^2_xw(x,t^*)=0,\;\;x\in I,
\ee
with $\,\partial_x^2w(\cdot,t^*),\;\mathcal H \partial^2_xw(\cdot,t^*)\in H^s(\R)$, $\,s>1/2$.

\vskip.1in Therefore, using Corollary \ref{col1} one has that $\,\partial_x^2w(\cdot,t^*)\equiv 0$ which implies that $\,w(\cdot,t^*)\equiv 0$ and completes the proof.

\end{proof}

 To extend the previous argument to prove Theorem \ref{TH2} we need the following result from complex analysis :
\begin{proposition}
\label{pro2}
Let  $\,J\subset [-\pi, \pi]$ be an open non-empty interval and 
$$
B_1(0)=\{z=x+iy\in \C : |z|<1\},\;\,A=\{ z\in \C : |z|=1,\,\arg(z)\in J\}.
$$

Let  $ \,F: B_1(0)\cup A\to \C$ be a continuous function  such that $\,F\big |_{B_1(0)}$ is analytic.

If $\,F\big |_A\equiv 0$, then $\,F\equiv 0$.

\end{proposition}

\begin{proof} The proof  follows from Proposition \ref{pro1} by considering $\,F_oT(z)$ where $\,T\,$ is a fractional linear transformation mapping the upper half-plane  to the unit disk $B_1(0)$.

\end{proof}

\section{Proof of Theorem \ref{TH5}}

First, we shall prove the following result :

\begin{corollary}
\label{col11}
Let $f\in H^s(\R),\,s>3/2$ be a real valued function. If there exists an open set $I\subset \R$ such that 
$$
f(x)=\mathcal L_{\delta}\partial_x f(x)=0,\;\;\;\;\;\;\;\forall \,x\in I,
$$
with $\,\mathcal L_{\delta} \,$ as in \eqref{T}, \eqref{symbol}, then $f\equiv 0$.

\end{corollary}

\begin{proof}
We define
\be\label{a1}
F(x)=\partial_xf(x)+i \mathcal L_{\delta}\partial_x f(x),\;\;\;\;x\in\R,
\ee
and consider its Fourier transform
\be\label{a2}
\begin{aligned}
\widehat{F}(\xi)&=\widehat{(\partial_xf+i\mathcal L_{\delta}\partial_xf)}(\xi)\\
&=2\pi i \xi (1+\rm{coth}(2\pi\delta \xi))\,\widehat{f}(\xi)\\
&=2\pi i\xi \Big(1+\frac{e^{2\pi\delta\xi}+e^{-2\pi\delta\xi}}{e^{2\pi\delta\xi}-e^{-2\pi\delta\xi}}\,\Big)\,\widehat{f}(\xi)\\
&=-4\pi i \xi \,\frac{e^{4\pi\delta \xi}}{1-e^{4\pi \delta \xi}}\,\widehat{f}(\xi)
\end{aligned}
\ee

We observe  that by considering $\,\partial_xf$ with $\,f\in H^s(\R), \,s>3/2,$ one cancels the singularity of $\,F\,$ at $\,\xi=0\,$ introduced by $\,\rm{coth}(\xi)$.

 By hypothesis and \eqref{a2} one concludes that $\,\widehat{F}\in L^1(\R)$ and has exponential decay for $\,\xi<0$. Hence, 
 \be\label{a3}
 F(x)=\int_{-\infty}^{\infty}\;e^{2\pi i \xi x}\,\widehat{F}(\xi)\,d\xi
 \ee
has an analytic extension 
\be\label{a4}
 F(x+iy)=\int_{-\infty}^{\infty}\;e^{2\pi i \xi (x+iy)}\,\widehat{F}(\xi)\,d\xi
 \ee
to the strip
$$
D_{2\delta}=\{z=x+iy\in \C\,:\,0<y<2\delta\}
$$
with $\,F\,$ continuous on $$
\,\{z=x+iy\,:\,0\leq y<2\delta\}$$ from the hypothesis on $\,f$. Now, Proposition \ref{pro1} leads the desired result.

\end{proof}

\begin{proof} [Proof of Theorem \ref{TH5}]

Once Corollary \ref{col11} is available the proof of Theorem \ref{TH5} is similar to that given for Theorem \ref{TH1}, therefore it will be omitted.

\end{proof}

\section{Proof of Theorem \ref{TH10}}

To prove Theorem \ref{TH10} we need an auxiliary lemma:

\begin{lemma}
\label{lemma1}
Let $\,f\in L^2(\R)$ be a real valued function. 
If there exists an open set $\,I \subset \R,\;0\in I,\,$ such that
\be\label{H12}
f(x,0)=0,\;\;\;\;x\in I,
\ee
and for each $\,N\in\Z^+$
\begin{equation}
\label{H11}
\int_{|x|\leq R}\;|\,\mathcal H f(x)|^2dx\leq c_N\,R^N\;\;\;\;\;\;\text{as}\;\;\;\;\;R \downarrow \,0,
\end{equation}
then,
\be\label{result10}
f(x)=0,\;\;\;\;x\in \R.
\ee

\end{lemma}

\begin{proof}
 Consider the analytic function $F=F(x+iy)$ defined  in $\,\R\times(0,\infty)$ with boundary values
$$F(x+i0)=-\mathcal Hf(x)+if(x).$$

Since $\,F\big|_{I}$ is real we can use Schwarz reflexion principle to find $\,\widetilde {F}$ analytic in $\,I\times (-\infty,\infty)\,$  with $\,\widetilde{F}=F$ on $\,I\times [0,\infty)$.
\vskip.07in
We observe : $\Re\,\widetilde{F}(x+i0)=\mathcal H f(x),\;x\in I\,$ with $\,\mathcal H f\big|_{I}\in C^{\infty}$, by the support property of $\,f$, and by assumption  \eqref{H11}  $\,\partial^j_x\mathcal H f(0)=0$, $\,j\in\Z^+\cup\{0\}$. Hence
$$
\frac{\partial^j}{\partial z^j}\,\widetilde {F}(0,0)=0\;\;\;\;\;\;\;j=0,1,2,....
$$
which completes the proof. 
\end{proof}

\vskip.1in

\begin{proof}[Proof of Theorem \ref{TH10}]
Defining $w(x,t)=(u_1-u_2)(x,t)$ it follows that
\be\label{eq11}
\partial_tw-\mathcal H \partial_x^2 w+\partial_xu_1 \,w+u_2\,\partial_xw=0,\;\;\;(x,t)\in\R\times[0,T].
\ee

Since $\,w(x,0)=0,\;x\in I$, one has that $\partial_x^jw(x,0)=0,\;x\in I$, $\,j\in \Z^+\cup\{0\}$, and using \eqref{eq11} 
$$
\mathcal H \partial_x^2 w(x,0)=\partial_tw(x,0)
$$
We now apply the hypothesis \eqref{H11} and Lemma \ref{lemma1} to conclude that $\,\partial_x^2w(x,0)=0,\;x\in\R$.

\end{proof}

\vspace{5mm}
\noindent\underline{\bf Acknowledgements.} C.E.K.  was supported by the NSF grant DMS-1800082. L.V. was supported by an
ERCEA Advanced Grant 2014 669689 - HADE, by the MINECO and by  BCAM Severo Ochoa excellence accreditation SEV-2013-0323.
project MTM2014-53850-P.


\begin{thebibliography}{9}

\bibitem{ABFS}  L. Abdelouhab, J. L. Bona, M. Felland, and J.-C.
Saut,  \emph{Nonlocal models for nonlinear dispersive waves},  Physica
D. {\bf 40} (1989) 360--392.

\bibitem{AbFo} M. J. Ablowitz and A. S. Fokas,
\emph {The inverse scattering transform for the Benjamin-Ono equation, a
pivot for multidimensional problems},
Stud. Appl. Math. {\bf 68} (1983) 1--10.


\bibitem{Be}  T. B. Benjamin, \emph{Internal waves of permanent
form in fluids of great depth},
J. Fluid Mech. {\bf 29} (1967) 559--592.



\bibitem{BiHu} J. Biello and J. K. Hunter, \emph{Nonlinear Hamiltonian waves with constant frequency and surface waves on vorticity discontinuities}, Comm. Pure Appl. Math. {\bf 63} 2009, 303--336.


  \bibitem{BuPl}  N. Burq and F. Planchon, \emph{On the well-posedness of the Benjamin-Ono equation},
Math. Ann. {\bf 340} (2008) 497--542.




\bibitem{CoWi} R. R. Coifman and M. Wickerhauser,
        \emph{The scattering transform for the Benjamin-Ono
equation}, Inverse Problems {\bf 6} (1990) 825--860.



\bibitem{FP} G. Fonseca and G. Ponce, \emph{The IVP for the Benjamin-Ono equation in weighted Sobolev spaces},  J. Funct. Anal. {\bf 260} (2010)  436--459.

\bibitem{FLP} G. Fonseca, F. Linares, and G. Ponce, \emph{The IVP for the Benjamin-Ono equation in weighted Sobolev spaces II},  J. Funct. Anal.
{\bf 262} (2012)  2031--2049.


\bibitem{GLM} Z. Guo, Y. Lin, and L. Molinet, \emph{Well-posedness in energy space for the periodic modified Benjamin?Ono equation}, 
J. Diff. Eqs.{\bf 256} (2014) 2778--2806.




\bibitem{IT} M. Ifrim and D. Tataru, \emph{Well-posedness and dispersive decay of small data solutions for the Benjamin-Ono equation},
pre-print 	arXiv:1701.08476


\bibitem{IoKe} A. D. Ionescu and C. E. Kenig, \emph{Global well-
posedness of the Benjamin-Ono equation on low-regularity spaces},
J. Amer. Math. Soc. {\bf 20}  (2007) 753--798.

\bibitem{Io1} R. J. Iorio, \emph{On the Cauchy problem for the
Benjamin-Ono equation},   Comm. Partial  Diff.  Eqs. {\bf 11}
  (1986) 1031--1081.

  \bibitem{Io2} R. J. Iorio, \emph{Unique continuation principle
for the Benjamin-Ono equation},   Diff. and Int. Eqs.  {\bf 16}
  (2003) 1281--1291.
  
  
  \bibitem{Jo} R. I. Joseph, \emph{Solitary waves in a finite depth fluid}, J. Phys. A 11 (1978) L97.

\bibitem{JE} R. I. Joseph,and R. Egri \emph{Multi-soliton solutions  in a finite depth fluid}, J. Phys. A 10 (1977) L225



\bibitem{Iz} V. Izakov
\emph{Carleman type estimates in an anisotropic case and applications,} J. Diff . Eqs. 
{\bf 105} (1993) 217--238.



\bibitem{KeKo} C. E. Kenig and K. D. Koenig, \emph{On the local
well-posedness of the Benjamin-Ono and  modified Benjamin-Ono equations},
Math. Res. Letters {\bf 10} (2003,)  879--895.


\bibitem{KeTa} C. E. Kenig and H. Takaoka, \emph{Global well-posedness of
the  modified Benjamin-Ono equation with initial data in $H^{1/2}$},
Int. Math. Res. Not. Art. ID 95702 (2006)  1--44.

\bibitem{KPV1} C. E. Kenig, G. Ponce, and L. Vega, \emph{On the generalized Benjamin-Ono equation}, Trans. Amer. Math. Soc. . {\bf 342} (1994) 155--172.



\bibitem{KoTz1} H. Koch and  N. Tzvetkov, \emph{On the local well-posedness of the  Benjamin-Ono equation on $H^{s}(\mathbb{R})$},
Int. Math. Res. Not. {\bf 26} (2003)  1449-1464.

\bibitem{KoTz2} H. Koch and  N. Tzvetkov, \emph{Nonlinear wave
interactions for the  Benjamin-Ono equation.}, Int. Math. Res.
Not.  {\bf 30} (2005)  1833--1847.



\bibitem{KSA} Y. Kodama, J. Satsuma and M.J. Ablowitz, \emph{Nonlinear intermediate long-wave equation: analysis and method of solution}, Phys.Rev. Lett. 46 (1981), 687-690.

\bibitem{KAS} Y. Kodama, M.J. Ablowitz and J. Satsuma, \emph{Direct and inverse scattering problems of
the nonlinear intermediate long wave equation}, J. Math. Physics 23 (1982), 564-576.



\bibitem{KKD} T. Kubota, D.R.S Ko and L.D. Dobbs, \emph{Weakly nonlinear, long internal gravity wavesin stratified fluids of finite depth}, J. Hydronautics {\bf 12} (1978), 157-165.





\bibitem{LiPo} F.  Linares and G. Ponce, \emph{Introduction to nonlinear dispersive equations}, second edition, Springer 
New York, 2014.

\bibitem{Mo1} L. Molinet \emph{Global well-posedness in the energy space for the Benjamin-Ono equation on the circle}, Math. Ann. {\bf 337} (2007) 353--383.

\bibitem{Mo2} L. Molinet \emph{Global well-posedness in $L^2$ for the periodic Benjamin-Ono equation}, Amer. J. Math. {\bf 130} (2008) 635--683.


\bibitem{Mo3} L. Molinet \emph{Sharp ill-posedness result for the periodic Benjamin-Ono equation}, J. Funct. Anal. {\bf 257} (2009),
348--3516.



\bibitem{MoPi} L. Molinet and D. Pilod, \emph{The Cauchy problem for the Benjamin-Ono equation in $L^2$ revisited}, Anal. PDE {\bf 5} (2012) 365--395.

\bibitem{MoRi1} L. Molinet and F. Ribaud, \emph{Well-posedness
results for the generalized Benjamin-Ono equation with small  initial
data}, J. Math. Pures et Appl. {\bf 83} (2004)  277--311.

\bibitem{MoRi2} L. Molinet and F. Ribaud, \emph{Well-posedness
results for the Benjamin-Ono equation with arbitrary large initial
data}, Int. Math. Res. Not. {\bf 70} (2004)  3757--3795.


\bibitem{MoRi3} L. Molinet and F. Ribaud, \emph{Well-posedness in $H^1$ for generalized Benjamin-Ono equations on the circle},  Discrete Contin. Dyn. Syst. {\bf 23} (2009)  1295--1311.

\bibitem{MoSaTz} L. Molinet, J.C. Saut, and  N. Tzvetkov, \emph{Ill-posedness issues for the  Benjamin-Ono and related equations},
SIAM J. Math. Anal. {\bf 33} (2001) 982--988.


\bibitem{MoVe} L. Molinet, and S. Vento, \emph{Improvement of the energy method for strongly nonresonant dispersive equations and applications}, Anal. PDE 8 (2015), no. 6, 1455--1495

\bibitem{On} H. Ono, \emph{Algebraic solitary waves on stratified
fluids},  J. Phy. Soc. Japan
 {\bf 39} (1975) 1082--1091.




\bibitem{Po} G. Ponce, \emph{On the global well-posedness of  the Benjamin-Ono equation},
Diff. and Int. Eqs. {\bf 4} (1991) 527--542.


\bibitem{Sa} J.-C. Saut, \emph{Benjamin-Ono and intermediate long 
wave equations : modeling, IST and PDE}, pre-print Fields Institute (2017)

\bibitem{SaSc} J.C. Saut and  B.Scheurer, \emph{Unique continuation for evolution equations,}
J. Diff. Eqs. \textbf{66} (1987), 118--137.


\bibitem{Ta} T. Tao, \emph{Global well-posedness of the   Benjamin-Ono equation on $H^{1}$}, Journal Hyp. Diff. Eqs. {\bf 1} (2004)
27--49.

\bibitem{Vent1} S. Vento, \emph{Sharp well-posedness results for the  generalized Benjamin-Ono equations  with higher nonlinearity}, Diff.  and Int. Eqs. {\bf 22} (2009) 425--446. 




\bibitem{Vent2} S. Vento, \emph{Well-posedness of the  generalized Benjamin-Ono equations  with arbitrary large initial data in the critical space}, Int. Math. Res. Not. {\bf 2} (2010)  297--319.


\bibitem{Wh} G. B. Whitham, \emph{Variational methods and applications to water waves}, Proc.R. Soc.
Lond. Ser. A., 299 (1967), 6-25.


%
\end{thebibliography}
\end{document}